\documentclass[reqno]{article}
\usepackage{amsthm}
\usepackage{hyperref}
\usepackage{amsfonts}
\usepackage{amsmath}
\usepackage{mathrsfs}
\usepackage{amssymb}
\usepackage[margin=1in]{geometry}
\theoremstyle{definition}
\newtheorem{theorem}{Theorem}
\newtheorem{lemma}[theorem]{Lemma}

\newtheorem{proposition}[theorem]{Proposition}
\newtheorem{corollary}[theorem]{Corollary}
\newtheorem{definition}{Definition}

\def\N{\mathbb{N}}
\def\L{\mathscr{L}}
\def\M{\mathcal{M}}
\def\FV{FV}
\newcommand{\claim}[1]{\vspace{2 mm}\noindent \textbf{Claim #1} \hspace{2 mm}}

\begin{document}

\title{A machine that knows its own code}
\author{Samuel A.~Alexander\thanks{Email:
alexander@math.ohio-state.edu}\hphantom{*}\\
\emph{Department of Mathematics, the Ohio State University}}
\date{June 2014}
\maketitle

\begin{abstract}
We construct a machine that knows its own code, at the price of
not knowing its own factivity.
\end{abstract}

\section{Introduction}

It is well known that a suitably idealized mechanical ``knowing agent'' capable of logic, arithmetic, and
self-reflection, cannot know the index of a 
Turing machine that represents its own knowledge.
See Lucas \cite{lucas}, Benacerraf \cite{benacerraf},
Reinhardt \cite{reinhardt}, Penrose \cite{penrose}, Carlson \cite{carlson2000},
and Putnam \cite{putnam}.
However, the proofs always involve (in various
guises) the machine knowing its own factivity:
that the machine satisfies $K(K\phi\rightarrow\phi)$.
We will relax this requirement and explicitly construct a machine that knows
its own code.
The construction resembles that of \cite{carlson2000} and \cite{carlson2012}.

Our result should be compared with that of Carlson \cite{carlson2000},
who showed that a truthful knowing agent can know its own truth and know that it
has \emph{some} code, without knowing which code.  A machine can know
its own factivity as well as that it has some
code (without knowing which), or it can know its own code exactly
but not know its own factivity (despite actually being factive).
This dichotomy in machine knowledge was first presented in the author's
dissertation \cite{alexanderdissert}.

In Section \ref{prelimsect}, we discuss preliminaries.

In Section \ref{constsect}, we construct a machine and prove that it knows its own code.


\section{Preliminaries}
\label{prelimsect}

We will work in the language $\L$ of Epistemic Arithmetic of S.\ Shapiro \cite{shapiro}.
This is the language of Peano arithmetic (with variables $x,y,z,\ldots$, constant symbol $0$,
unary function symbol $S$ for successor,
and binary function symbols $+$ and $\cdot$ for addition and multiplication), extended
by a modal operator $K$ for knowledge.
The well-formed formulas of $\L$ (and their free variables
$\phi\mapsto FV(\phi)$) are defined in the usual way; a formula of the form
$K(\phi)$ is called \emph{purely modal}, and will be written $K\phi$ if no confusion results.
Formulas without free variables are \emph{sentences}.
Terms, substitutability,
and the result $\phi(x|t)$ of substituting term $t$ for variable $x$ in $\phi$, are defined in the 
obvious ways.

We borrow the following semantics from T.J.\ Carlson \cite{carlson2000} (pp.\ 54--55).
We have reworded the definition in an equivalent form (except that Carlson allowed
for multiple operators while we need only one).
The intuition is that purely modal formulas should be
treated as much like propositional atoms as possible.

\begin{definition}
\label{baselogicdef}
(The Base Logic)
\begin{enumerate}
\item
If $U$ is some set, an \emph{assignment into $U$} is a function that maps variables of $\L$
into $U$.
\item If $s$ is an assignment into $U$, $x$ is a variable, and $u\in U$, $s(x|u)$ shall
mean the assignment into $U$ that agrees with $s$ except that it maps $x$ to $u$.
\item
An $\L$-\emph{structure} $\M$ consists of a first-order structure $\M_0$ for the 
first-order part
of $\L$, together with a function that takes one assignment $s$ (into the universe of $\M_0$)
and one purely modal formula $K\phi$, and outputs either True or False---in which case we 
write
$\M\models K\phi[s]$ or $\M\not\models K\phi[s]$, respectively---satisfying
the following three constraints:
\begin{enumerate}
\item Whether or not $\M\models K\phi[s]$
does not depend on $s(x)$ if $x$ is not a free variable of $\phi$.
\item If $\psi$ is an \emph{alphabetic variant}
of $\phi$ (meaning that $\psi$ is obtained from $\phi$ by
renaming bound variables so as to respect the
binding of the quantifiers) then, for any assignment $s$,
$\M\models K\phi[s]$ if and only if $\M\models K\psi[s]$.
\item (Weak Substitution)\footnote{The full Substitution Lemma,
where variable $y$ is replaced by an arbitrary term $t$,
is not generally valid in modal logic.} If $x$ and $y$ are variables, $K\phi$ is a modal formula,
$y$ is substitutable for $x$ in $\phi$,
and $s$ is an assignment, then $\M\models K\phi(x|y)[s]$
if and only if $\M\models K\phi[s(x|s(y))]$.
\end{enumerate}
\item From this, for any formula $\phi$, $\M\models \phi[s]$ and $\M\not\models\phi[s]$ are defined
in the usual inductive way.  We say $\M\models\phi$ if $\M\models\phi[s]$ for every assignment 
$s$.
\item If $\Sigma$ is a set of $\L$-sentences and $\phi$ is
an $\L$-formula, we write $\Sigma\models \phi$
to indicate that for every $\L$-structure $\M$, if $\M\models \Sigma$ (meaning $\M\models \sigma$ 
for every $\sigma\in\Sigma$) then $\M\models\phi$.
\item An $\L$-formula $\phi$ is \emph{valid} if $\emptyset\models\phi$.
\end{enumerate}
\end{definition}

\begin{lemma}
\label{lemma1}
(Completeness and compactness)
\begin{enumerate}
\item The set of valid $\L$-formulas is r.e.
\item For any r.e.\ set $\Sigma$ of $\L$-sentences,
$\{\phi\,:\,\Sigma\models\phi\}$ is r.e.
\item There is an effective procedure that, given (a G\"{o}del number of) an r.e.\ 
set $\Sigma$ of $\L$-sentences, outputs (a G\"{o}del number of)
$\{\phi\,:\,\Sigma\models\phi\}$.
\item If $\Sigma$ is a set of $\L$-sentences and $\Sigma\models\phi$,
there is a finite set $\sigma_1,\ldots,\sigma_n\in\Sigma$
such that\footnote{Throughout
the paper, $A\rightarrow B\rightarrow C$
is shorthand for $A\rightarrow (B\rightarrow C)$, and similar for longer implication chains.} 
$\sigma_1\rightarrow\cdots\rightarrow\sigma_n\rightarrow\phi$ is 
valid.
\end{enumerate}
\end{lemma}

\begin{proof}
Straightforward.
\end{proof}

\begin{definition}
The \emph{axioms of Peano arithmetic for $\L$}
consist of the axioms of Peano arithmetic, with the induction schema extended to
$\L$.  To be precise, the axioms of Peano arithmetic for $\L$ are as follows.
\begin{enumerate}
\item $\forall x(S(x)\not=0)$.
\item $\forall x\forall y(S(x)=S(y)\rightarrow x=y)$.
\item $\forall x(x+0=x)$.
\item $\forall x\forall y(x+S(y)=S(x+y))$.
\item $\forall x(x\cdot 0=0)$.
\item $\forall x\forall y(x\cdot S(y)=x\cdot y + x)$.
\item The universal closure of $\phi(x|0)\rightarrow (\forall x(\phi\rightarrow \phi(x|S(x))))
\rightarrow \forall x \phi$ for any $\L$-formula $\phi$.
\end{enumerate}
\end{definition}

\begin{definition}
\label{axiomsomnibus}
\item
\begin{itemize}
\item
The \emph{pre-closure axioms of knowledge} are given by the following schemata.
\begin{itemize}
\item $E1$: The universal closure of $K\phi$ whenever $\phi$ is valid.
\item $E2$: The universal closure of $K(\phi\rightarrow\psi)\rightarrow K\phi\rightarrow 
K\psi$.
\item $E3$: The universal closure of $K\phi\rightarrow\phi$.
\item $E4$: The universal closure of $K\phi\rightarrow KK\phi$.
\end{itemize}
\item
The \emph{axioms of knowledge} consist of the pre-closure axioms of knowledge along with $K\phi$
whenever $\phi$ is a pre-closure axiom of knowledge.
\item
The axioms of \emph{epistemic arithmetic} consist of the pre-closure axioms of knowledge along 
with $K\phi$ 
whenever $\phi$ is a pre-closure axiom of knowledge or $\phi$ is an axiom of Peano arithmetic for 
$\L$.
\item
The \emph{axioms of knowledge mod factivity} consist of the pre-closure axioms of knowledge along 
with $K\phi$ whenever $\phi$ is an instance of $E1$, $E2$, or $E4$.
\item
The axioms of \emph{epistemic arithmetic mod factivity}
consist of the pre-closure axioms of knowledge along with $K\phi$ whenever $\phi$ is an instance 
of $E1$, $E2$, $E4$, or an axiom of Peano arithmetic for $\L$.
\end{itemize}
\end{definition}

\begin{definition}
By \emph{Reinhardt's schema} we mean the following schema (\cite{reinhardt}, p.~327)
\begin{itemize}
\item $\exists e K\forall x(K\phi\leftrightarrow x\in W_e)$, whenever
$FV(\phi)\subseteq\{x\}$.
\end{itemize}
\end{definition}

Reinhardt demonstrated that a formalization of ``I am a Turing machine and I know which one'' 
cannot be consistent with epistemic arithmetic.  To do this, he found a 
particular instance of Reinhardt's schema that was inconsistent with epistemic arithmetic.
A truthful mechanical knowing agent that knows its own code necessarily knows all
instances of Reinhardt's schema (for example,
suppose $\phi$ is the formula ``the $x$th Turing machine runs forever'';
if I know my own code, I can deduce a code
for the set of those $n\in\N$ such that I know the $n$th Turing machine runs forever).

We will show that Reinhardt's schema is consistent (in fact, $\omega$-consistent, by which we
mean it has a structure with universe $\N$ where the symbols of Peano arithmetic are given the 
usual interpretations) with epistemic arithmetic mod factivity.

This result should be compared with the main result of \cite{carlson2000}.
Along with the above schema, Reinhardt introduced (\cite{reinhardt}, p.~320) a weaker schema, 
Reinhardt's
\emph{strong 
mechanistic thesis}, $K\exists e\forall x(K\phi\leftrightarrow x\in W_e)$.\footnote{Reinhardt 
lists only $\exists e\forall x(K\phi\leftrightarrow x\in W_e)$,
and $K\exists e\forall x(K\phi\leftrightarrow x\in W_e)$ follows by the rule of
necessitation.  Clearly the latter formula is what is important.
Reinhardt originally referred to this as the Post-Turing thesis, and later
decided on the name \emph{strong mechanistic thesis} (see \cite{carlson2000} p.~54).}
Reinhardt conjectured, and Carlson proved\footnote{This was accomplished using deep structural 
theorems on the ordinal
numbers \cite{carlson1999}, later organized into patterns of resemblance \cite{carlson2001}.}, that 
the strong mechanistic thesis is consistent
with epistemic arithmetic.
Thus we have a dichotomy: a truthful knowing machine can know it is \emph{some} machine (but not 
which one), and also know itself to be truthful; alternatively, a truthful knowing machine can know 
precisely which machine it is, but not know itself to be truthful.

\section{The Construction}
\label{constsect}

\begin{definition}
Suppose $\phi$ is an $\L$-sentence and $s$ is an assignment into $\N$.
We define $\phi^s$ to be the sentence
\[
\phi^s=\phi(x|\overline{s(x)})(y|\overline{s(y)})\cdots
\]
obtained by replacing each free variable in $\phi$ by a numeral for
the natural number it is assigned to.
\end{definition}

For example, if $s(x)=0$ and $s(y)=2$, then $(x=y)^s$ is
the sentence $(0=S(S(0)))$.

The machine we construct will have the following form for a certain well-chosen set $\Sigma$.

\begin{definition}
\label{Msubsigmadef}
If $\Sigma$ is a set of $\L$-sentences, let $\M_\Sigma$ be the $\L$-structure with
universe $\N$, in which symbols of Peano arithmetic are interpreted in the usual way,
and in which knowledge is interpreted so that for all $\L$-formulas $\phi$ and
assignments $s$ into $\N$,
\[
\mbox{$\M_\Sigma\models K\phi[s]$ iff $\Sigma\models\phi^s.$}
\]
\end{definition}

\begin{lemma}
For any $\Sigma$ as in Definition \ref{Msubsigmadef}, $\M_\Sigma$
really is an $\L$-structure.
\end{lemma}

\begin{proof}
We must verify the conditions on $\M_\Sigma\models K\phi[s]$ from Definition 
\ref{baselogicdef}.  Let $s$ be an assignment into $\N$.

\begin{itemize}
\item (a)
If $x$ is not free in $\phi$, then $\phi^s$ does not depend on $s(x)$, so neither does
$\Sigma\models\phi^s$, so neither does $\M_\Sigma\models K\phi[s]$.
\item (b)
An easy inductive argument shows that any time $\psi$ is an alphabetic variant of $\phi$,
for any assignment $s$ into $\N$, $\psi^s$ is an alphabetic variant of $\phi^s$.
Another easy induction shows that whenever $\psi$ is an alphabetic variant of $\phi$,
$\psi\leftrightarrow\phi$ is valid, so certainly $\Sigma\models \phi\leftrightarrow\psi$.
It follows that (when $\psi$ is an alphabetic variant of $\phi$)
$\M_\Sigma\models K\phi[s]$ if and only if $\M_\Sigma\models K\psi[s]$.
\item (c)
(Weak Substitution)
Let $x$ and $y$ be variables.
An easy inductive argument shows that for all assignments $t$ into $\N$
and all formulas $\phi$ such that $y$ is substitutable for $x$ in $\phi$, $\phi(x|y)^t\equiv 
\phi^{t(x|t(y))}$.
By definition $\M_\Sigma\models K\phi(x|y)[s]$
if and only if $\Sigma\models \phi(x|y)^s$, which holds if and only if
$\Sigma\models \phi^{s(x|s(y))}$,
which is true if and only if $\M_\Sigma\models K\phi[s(x|s(y))]$.
\end{itemize}
\end{proof}

\begin{lemma}
\label{claim1}
For any $\Sigma$ as in Definition \ref{Msubsigmadef},
any $\L$-formula $\phi$, and any assignment $s$,
$\M_\Sigma\models\phi[s]$ if and only if $\M_\Sigma\models\phi^s$.
\end{lemma}

\begin{proof}
By induction on formula complexity of $\phi$.
The most interesting case is when $\phi$ is $K\phi_0$ for some
formula $\phi_0$.
Suppose $\M_\Sigma\models K\phi_0[s]$,
so $\Sigma\models\phi_0^s$.
If we let $t$ be an arbitrary assignment, since $\phi_0^s$ is a sentence,
$\phi_0^s\equiv (\phi_0^s)^t$ and thus $\Sigma\models (\phi_0^s)^t$.
By definition this means $\M_\Sigma\models K\phi_0^s[t]$.
By arbitrariness of $t$, $\M_\Sigma\models K\phi_0^s$.
The converse is similar.
\end{proof}

\begin{lemma}
\label{e2lem}
For any $\Sigma$ as in Definition \ref{Msubsigmadef},
$\M_\Sigma$ satisfies all instances of $E2$.
\end{lemma}

\begin{proof}
Let $s$ be an assignment and 
suppose $\M_\Sigma\models K(\phi\rightarrow\psi)[s]$
and $\M_\Sigma\models K\phi[s]$.
This means $\Sigma\models (\phi\rightarrow\psi)^s$
and $\Sigma\models \phi^s$.
Clearly $(\phi\rightarrow\psi)^s\equiv \phi^s\rightarrow\psi^s$,
so by modus ponens, $\Sigma\models \psi^s$,
so $\M_\Sigma\models \psi[s]$.
\end{proof}

\begin{lemma}
\label{palem}
For any $\Sigma$ as in Definition \ref{Msubsigmadef},
$\M_\Sigma$ satisfies the axioms of Peano arithmetic for $\L$.
\end{lemma}

\begin{proof}
Let $\M=\M_\Sigma$.
Let $\psi$ be an axiom of Peano arithmetic.  If $\psi$ is any other axiom
besides an instance of induction, $\M\models\psi$ because $\M$ has universe
$\N$ and interprets the symbols of Peano arithmetic in the intended ways.
But suppose $\psi$ is a universal closure of
\[
\phi(x|0)\rightarrow (\forall x(\phi\rightarrow \phi(x|S(x))))\rightarrow\forall x\phi.
\]
Let $s$ be an assignment and assume $\M\models \phi(x|0)[s]$ and $\M\models \forall 
x(\phi\rightarrow \phi(x|S(x)))[s]$.  We must show $\M\models \forall x\phi[s]$.

Since $\M\models \phi(x|0)[s]$, Lemma \ref{claim1} says $\M\models \phi(x|0)^s$.
Clearly $\phi(x|0)^s\equiv \phi^{s(x|0)}$, so $\M\models \phi^{s(x|0)}$.

For each $m\in\N$, since $\M\models \forall x(\phi\rightarrow \phi(x|S(x)))[s]$,
in particular $\M\models\phi\rightarrow\phi(x|S(x))[s(x|m)]$.
And thus, \emph{if} $\M\models \phi[s(x|m)]$, then $\M\models \phi(x|S(x))[s(x|m)]$.
By Lemma \ref{claim1}, that last sentence can be rephrased:
\emph{if} $\M\models \phi^{s(x|m)}$, then $\M\models \phi(x|S(x))^{s(x|m)}$;
but clearly $\phi(x|S(x))^{s(x|m)}\equiv \phi^{s(x|m+1)}$, so in summary:
\begin{itemize}
\item $\M\models \phi^{s(x|0)}$.
\item For each $m\in\N$, if $\M\models \phi^{s(x|m)}$, then
$\M\models \phi^{s(x|m+1)}$.
\end{itemize}
Therefore, by mathematical induction, $\M\models \phi^{s(x|m)}$ for every $m\in\N$.
By Lemma \ref{claim1}, for all $m\in\N$, $\M\models \phi[s(x|m)]$.
So $\M\models \forall x\phi[s]$, as desired.
\end{proof}

\begin{lemma}
\label{e4lem}
Suppose $\Sigma$ (as in Definition \ref{Msubsigmadef}) is
\emph{closed under $K$}, by which we mean that for every $\phi\in\Sigma$,
$K\phi\in\Sigma$.
Furthermore, assume $\Sigma$ contains all instances of $E1$ and $E2$ from Definition
\ref{axiomsomnibus}.
Then $\M_\Sigma$ satisfies all instances of $E4$.
\end{lemma}

\begin{proof}
Assume $\M_\Sigma\models K\phi[s]$.
This means $\Sigma\models \phi^s$.
By Lemma \ref{lemma1} there are finitely many $\sigma_1,\ldots,\sigma_n\in\Sigma$
such that $\sigma_1\rightarrow\cdots\rightarrow \sigma_n\rightarrow \phi^s$
is valid.
Thus,
the universal closure of
$K(\sigma_1\rightarrow\cdots\rightarrow \sigma_n\rightarrow \phi^s)$
is an instance of $E1$, hence in $\Sigma$.
By repeated instances of $E2$ in $\Sigma$, $\Sigma$ implies the universal closure of
\[
K(\sigma_1\rightarrow\cdots\rightarrow \sigma_n\rightarrow \phi^s)
\rightarrow K\sigma_1\rightarrow \cdots \rightarrow K\sigma_n\rightarrow K\phi^s.\]
It follows that $\Sigma\models K\phi^s$,
so $\M_\Sigma\models KK\phi[s]$.
\end{proof}

\begin{definition}
By \emph{assigned validity} we mean the following schemata of $\L$-sentences:
\begin{itemize}
\item $\phi^s$, whenever $\phi$ is valid and $s$ is any assignment.
\end{itemize}
\end{definition}

\begin{lemma}
\label{assignedvaliditylem}
For any $\Sigma$ as in Definition \ref{Msubsigmadef},
if $\Sigma$ contains all instances of assigned validity,
then $\M_\Sigma$ satisfies all instances of $E1$.
\end{lemma}

\begin{proof}
Suppose $\phi$ is valid and $s$ is any assignment, we will show
$\M_\Sigma\models K\phi[s]$.
Since $\phi$ is valid, $\phi^s$ is an instance of assigned validity,
so $\Sigma\models\phi^s$ by assumption.
Thus $\M_\Sigma\models K\phi[s]$.
\end{proof}

\begin{definition}
\label{sigmandef}
For every $n\in\N$, let $\Sigma(n)$ be the family of axioms consisting of the following 
$\L$-schemata.
\begin{enumerate}
\item $E1$, $E2$, and $E4$.
\item The axioms of Peano arithmetic for $\L$.
\item $\forall x (K\phi\leftrightarrow \langle 
x,\overline{\ulcorner\phi\urcorner}\rangle\in 
W_{\overline{n}})$,
$\phi$ any $\L$-formula with $FV(\phi)\subseteq\{x\}$.
\begin{itemize}
\item Here $\ulcorner\bullet\urcorner$ denotes canonical G\"{o}del number,
$\overline{\bullet}$ denotes numeral,
and
$\langle \bullet,\bullet\rangle$ abbreviates a definition (in Peano arithmetic)
of a canonical computable bijection $\N^2\to\N$.
\end{itemize}
\item Assigned validity.
\item $K\phi$, whenever $\phi$ is an instance of any of lines 1--4 or (recursively) 5.
\end{enumerate}
\end{definition}

\begin{lemma}
\label{sigmanbound}
For every $n\in\N$ and every $\phi\in\Sigma(n)$, $\phi$ is a sentence.
\end{lemma}

\begin{proof}
By inspection.
\end{proof}

\begin{lemma}
\label{ctthesis}
There is a total computable function $f:\N\to\N$ such that for every $n$,
\[
W_{f(n)}=\{\langle m,\ulcorner\phi\urcorner\rangle\in \N\,:\,
\mbox{$\phi$ is a formula with $FV(\phi)\subseteq\{x\}$
and $\Sigma(n)\models \phi(x|\overline{m})$}\}.
\]
\end{lemma}

\begin{proof}
Follows from Lemma \ref{lemma1} and the Church-Turing Thesis.
\end{proof}

\begin{corollary}
\label{kleene}
There is an $n\in\N$ such that
\[
W_n = \{\langle m,\ulcorner\phi\urcorner\rangle\in\N
\,:\,
\mbox{$\phi$ is a formula with $FV(\phi)\subseteq \{x\}$ and
$\Sigma(n)\models \phi(x|\overline{m})$}\}.
\]
\end{corollary}

\begin{proof}
By Kleene's Recursion Theorem and Lemma \ref{ctthesis}.
\end{proof}

\begin{proposition}
\label{bigproposition}
Let $n$ be as in Corollary \ref{kleene}.
Then $\M_{\Sigma(n)}\models \Sigma(n)$.
\end{proposition}

\begin{proof}
For brevity, write $\Sigma$ for $\Sigma(n)$ and $\M$ for $\M_{\Sigma(n)}$.

\item
\claim{1}
$\M$ satisfies all instances of $E1$.  By Lemma \ref{assignedvaliditylem}.

\item
\claim{2}
$\M$ satisfies all instances of $E2$.  By Lemma \ref{e2lem}.

\item
\claim{3}
$\M$ satisfies all instances of $E4$.  By Lemma \ref{e4lem}.

\item
\claim{4}
$\M$ satisfies the axioms of Peano arithmetic for $\L$.
By Lemma \ref{palem}.

\item
\claim{5}
For any $\L$-formula $\phi$ with $FV(\phi)\subseteq\{x\}$,
$\M$ satisfies $\forall x(K\phi\leftrightarrow \langle 
x,\overline{\ulcorner\phi\urcorner}\rangle\in W_{\overline{n}})$.

Let $s$ be an arbitrary assignment (say with $s(x)=m$), we must show $\M\models K\phi[s]$
if and only if $\M\models \langle
x,\overline{\ulcorner\phi\urcorner}\rangle\in W_{\overline{n}}[s]$.
The following are equivalent:
\begin{align*}
\M &\models K\phi[s]\\
\Sigma &\models \phi^s &\mbox{(Definition of $\M$)}\\
\Sigma &\models \phi(x|\overline{m}) &\mbox{(Since $\FV(\phi)\subseteq\{x\}$)}\\
\langle m,\ulcorner\phi\urcorner\rangle &\in W_n
&\mbox{(By choice of $n$ (Corollary \ref{kleene}))}\\
\M &\models \langle \overline{m},\overline{\ulcorner\phi\urcorner}\rangle\in W_{\overline{n}}
&\mbox{(Since $\M$ has standard first-order part)}\\
\M &\models (\langle x,\overline{\ulcorner\phi\urcorner}\rangle\in W_{\overline{n}})^s
&\mbox{(Since $s(x)=m$)}\\
\M &\models \langle x,\overline{\ulcorner\phi\urcorner}\rangle\in W_{\overline{n}}[s].
&\mbox{(By Lemma \ref{claim1})}
\end{align*}

\item
\claim{6}
$\M$ satisfies all instances of assigned validity.
Suppose $\phi$ is valid and $s$ is an assignment, we must show $\M\models \phi^s$.
By Lemma \ref{claim1}, it suffices to show $\M\models\phi[s]$.
But this is immediate, because $\phi$ is valid.

\item
\claim{7}
$\M\models K\phi$ whenever $K\phi$ is an instance of line 5 from Definition
\ref{sigmandef}.
For any such $K\phi$, $\phi$ itself lies in $\Sigma$,
so $\Sigma\models\phi$.
Let $s$ be any assignment.
By Lemma \ref{sigmanbound}, $\phi$ is a sentence, thus $\phi^s=\phi$
and so $\Sigma\models\phi^s$, meaning $\M\models K\phi[s]$.
\end{proof}

\begin{theorem}
Let $n$ be as in Corollary \ref{kleene}.
\begin{enumerate}
\item $\M_{\Sigma(n)}$ satisfies the axioms of epistemic arithmetic mod factivity.
\item $\M_{\Sigma(n)}$ satisfies all instances of Reinhardt's schema, that is,
\[\exists e K\forall x(K\phi\leftrightarrow x\in W_e)\] whenever
$\FV(\phi)\subseteq\{x\}$.
\item Additionally, there is a fixed $m\in\N$ such that $\M_{\Sigma(n)}$
satisfies the schema $K(K\phi\leftrightarrow \ulcorner\phi\urcorner\in W_{\overline{m}})$,
where $\phi$ ranges over $\L$-sentences.
\end{enumerate}
Thus, the machine that knows the things known by $\M_{\Sigma(n)}$
is a machine that knows its own code.
\end{theorem}

\begin{proof}
\item
(1). The only axiom schema that remains to be proven is $E3$, the universal closures of formulas of 
the form $K\phi\rightarrow\phi$.  Suppose $s$ is any assignment and $\M_{\Sigma(n)}\models 
K\phi[s]$.
This means $\Sigma(n)\models\phi^s$.  By Proposition \ref{bigproposition},
$\M_{\Sigma(n)}\models\Sigma(n)$, therefore $\M_{\Sigma(n)}\models\phi^s$.
By Lemma \ref{claim1}, $\M_{\Sigma(n)}\models\phi[s]$, establishing (1).

\item
(2) and (3). By combining lines 3 and 5 of Definition \ref{sigmandef},
$\Sigma(n)$ contains $K\forall x(K\phi\leftrightarrow
\langle x,\overline{\ulcorner\phi\urcorner}\rangle\in W_{\overline{n}})$ whenever $\phi$
is an $\L$-formula with $FV(\phi)\subseteq\{x\}$.
(2) and (3) follow.
\end{proof}

\section{Conclusion and Related Work}

A knowing machine (implicitly meaning, a knowing machine
that knows its own factivity) cannot know its own code.
Carlson showed that such a machine can know that it has \emph{some} code (without
knowing exactly which).  Our result complements Carlson's: it is possible for a machine
to know its code quite precisely, at the price of knowing its factivity
(despite really being factive).

In our dissertation \cite{alexanderdissert} we explore related issues surrounding
multiple interacting machines.
Suppose $\prec$ is an r.e.~well-founded partial ordering of $\N$.
\begin{itemize}
\item There are machines $M_0,M_1,\ldots$ such that
each $M_i$ knows precise codes of each $M_j$, and knows factivity of $M_j$
when $j\prec i$.
\item There are machines $M_0,M_1,\ldots$ such that
each $M_i$ knows precise codes of $M_j$ when $j\prec i$;
factivity of $M_j$ when $j\preceq i$; and each $M_i$ knows that each $M_j$
has some code (without necessarily knowing which).
\item There are machines $M_0,M_1,\ldots$ such that
each $M_i$ knows precise codes of $M_j$ ($j\prec i$);
factivity of $M_j$ ($j\preceq i$); a slight weakening of factivity of $M_j$ (all $j$);
and that $M_j$ has some code ($j\preceq i$).
\item But if $\prec$ is ill-founded,
there are no such machines as above, provided the machines are also required
to know rudimentary facts about computable ordinals.
\end{itemize}
We are preparing a streamlined paper on these
results for journal submission.

\end{document}